\documentclass{article}

\usepackage{arxiv}

\usepackage[utf8]{inputenc} 
\usepackage[T1]{fontenc}    
\usepackage{hyperref}       
\usepackage{url}            
\usepackage{booktabs}       
\usepackage{amsfonts}       
\usepackage{nicefrac}       
\usepackage{microtype}      
\usepackage{lipsum}		
\usepackage{graphicx}
\usepackage{mathptmx, amsmath, amssymb, amsfonts, amsthm, mathptmx, enumerate, color}

\numberwithin{equation}{section}
\newtheorem{theorem}{Theorem}[section]
\newtheorem{lemma}[theorem]{Lemma}

\newtheorem{remark}[theorem]{Remark}
\newtheorem{definition}[theorem]{Definition}
\newtheorem{example}[theorem]{Example}

\title{Some
properties of Sadik transform and its applications of fractional-order
dynamical systems in control theory}

\author{
  Saleh S. Redhwan \thanks{ S.S. Redhwan, saleh.redhwan909@gmail.com.} \\
 Department of Mathematics\\
 Dr.Babasaheb Ambedkar Marathwada\\
University, Aurangabad, (M.S),431001, India\\
  \texttt{saleh.redhwan909@gmail.com} \\
   \And
 Sadikali L. Shaikh\\
  Department of Mathematics\\
 Maulana Azad College of arts, \\
 Science and Commerce \\
 Aurangabad, (M.S),431001, India\\
  \texttt{} \\
   \AND
  Mohammed S. Abdo$^1$$^,$$^2$\\
  $^1$Department of Mathematics\\
 Dr.Babasaheb Ambedkar Marathwada  University\\
Aurangabad, (M.S),431001, India\\
 $^2$Hodeidah University\\
 Al-Hodeidah, 3114, Yemen\\
  \texttt{msabdo1977@gmail.com} \\
}

\begin{document}
\maketitle

\begin{abstract}
In this paper, we study some new properties of Sadik transform such as
integration, time delay, initial value theorem, and final value theorem.
Moreover, we prove the theorem of Sadik transform for Caputo fractional
derivative and we also establish sufficient conditions for the existence of
the Sadik transform of Caputo fractional derivatives. At the end, the
fractional-order dynamical systems in control theory as application of this
transform is discussed, in addition, some numerical examples to justify our
results.
\end{abstract}


\section{Introduction}

The integral transformation method is excessively used to solve different
kinds of differential equations in a simple way. As integral transforms
converts differential equations into algebraic equation algebraic equations
are more simple than differential equations. In the literature, there are
several integral transforms and all of them are acceptable to resolve
numerous types differential equations. Recently some new integral transforms
were introduced, see \cite{HA,KE1,KE2,TE}. Recently Shaikh \cite%
{SSa1,Sa,SSa2,SSa3}, proposed a new integral transform that known as Sadik
transform. This transform is an unification of some famous transforms such
as Laplace transform, Sumudu transform, Kamal transform, Tarig and Laplace-
Carson transform and Elzaki transform. For example, Shaikh in \cite{Sa},
presented some properties of this transform, like the existing theorem of
Sadik transform and duality theorem. Further, the author proved that above
mentioned transforms are particular cases of Sadik transform. Shaikh in \cite%
{SSa2}, proved properties of Sadik transform for derivative of functions,
shifting theorem for Sadik transform. Also, the author obtained transfer
function of dynamical system in control theory using Sadik transform.
Moreover, he solved some applications in control theory by Sadik transform.
At the outsight, integral transform method is useful and effective tool for
solving fractional differential equations. But it is also true that all
types of fractional differential equations are not solvable by integral
transform technique see \cite{Ju,Ke,Gu} and the references therein. Abhale
and Pawar in \cite{NS}, proved some fundamental properties of Sadik
transform and used these properties of Sadik transform to solve first order
and second order ordinary differential equations.

On the other hand, the fractional calculus is a generalization of classical
differentiation and integration into non-integer order. Some fundamental
definitions of fractional derivatives were given by Riemann-Liouville,
Hadamard, Caputo, Hilfer, Liouville-Caputo, Gr\"{u}nwald-Letnikov, Riesz,
Coimbra and Weyl. The fractional derivatives describe the property of memory
and heredity of many materials. see \cite{KL1,KB}.

Fractional ordinary and partial differential equations, as a generalization
of classical integer order differential equations. Fractional differential
equations have enabled the investigation of multiple phenomena such as
diffusion processes, electrodynamics, fluid flow, elasticity and it
increasingly used to model problems in biology, viscoelasticity, fluid
mechanics, physics, engineering, and others applications \cite{KL1,KB,KJ,IP}.

In this paper, we introduce Sadik transform of fractional order (Caputo
derivative operator) and some new properties of this transform such as time
delay, initial and final value theorems are proved. Further, we established
a sufficient condition for the existence of Sadik transform of Caputo
fractional derivatives and have solved Caputo fractional differential
equations. Finally, the time-domain analysis of dynamical systems involving
fractional-order is presented to solving systems of control theory.

\section{preliminaries:}

In this section, we recall some notions, definitions and lemmas that used
through this paper. Let $[a,b]\subset
\mathbb{R}
^{+}$ and $\mathcal{C}[a,b]$ be the space of all continuous functions $%
\varphi :[a,b]\longrightarrow
\mathbb{R}
$ with the norm $\left\Vert \varphi \right\Vert _{\infty }=\max \{\left\vert
\varphi (t)\right\vert :$ $t\in \lbrack a,b]\}$ for any $\varphi \in
\mathcal{C}[a,b].$ Denote $L^{1}[a,b]$ the labesgue integrable functions
with the norm $\left\Vert \varphi \right\Vert
_{L^{1}}=\int_{a}^{b}\left\vert \varphi (t)\right\vert dt<\infty .$

\begin{definition}
\cite{KL1} Let $\varphi $ be a locally integrable function on $[0,+\infty )$%
. The Riemann-Liouville fractional integral of order $\gamma >0$ of function
$\varphi $ is given by
\begin{equation*}
I_{0^{+}}^{\gamma }\varphi (t)=\left\{
\begin{array}{c}
\frac{1}{\Gamma (\gamma )}\int_{0}^{t}(t-\tau )^{\gamma -1}\varphi (\tau
)d\tau ,\text{ }\qquad \gamma >0, \\
\varphi (t),\text{\qquad \qquad \qquad \qquad \qquad }\gamma =0.%
\end{array}%
\right.
\end{equation*}%
where $\Gamma (z)=\int_{0}^{\infty }e^{-t}t^{z-1}dt$ is a Gamma function of
Euler, $z\in
\mathbb{C}
.$
\end{definition}

\begin{definition}
\cite{KL1} Let $n-1<\gamma <n\in
\mathbb{N}
,$ and $\varphi (t)$ has absolutely continuous derivatives up to order $%
(n-1).$ Then the left sided Riemann-Liouville fractional derivative of order
$\gamma $ of $\varphi $ is defined by
\begin{eqnarray*}
D_{0^{+}}^{\gamma }\varphi (t) &=&\left( \frac{d}{dt}\right)
^{n}I_{0^{+}}^{n-\gamma }\varphi (t) \\
&=&\left( \frac{d}{dt}\right) ^{n}\frac{1}{\Gamma (n-\gamma )}%
\int_{0}^{t}(t-\tau )^{n-\gamma -1}\varphi (\tau )d\tau ,
\end{eqnarray*}%
where $n=[\gamma ]+1$ and $[\gamma ]$ denotes the integer part of the real
number $\gamma .$
\end{definition}

\begin{definition}
\cite{KL1} \label{d1} The left sided Caputo derivative of fractional order $%
\gamma $ $(n-1<\gamma <n\in
\mathbb{N}
)$ is given by
\begin{equation*}
^{c}D_{0^{+}}^{\gamma }\varphi (t)=I_{0^{+}}^{n-\gamma }\frac{d^{n}}{dt^{n}}%
\varphi (t)=\frac{1}{\Gamma (n-\gamma )}\int_{0}^{t}(t-\tau )^{n-\gamma
-1}\varphi ^{(n)}(\tau )d\tau ,
\end{equation*}%
where the function $\varphi (t)$ has absolutely continuous derivatives up to
order $(n-1)$. In particular, if $0<\gamma <1$, we have%
\begin{equation*}
^{c}D_{0^{+}}^{\gamma }\varphi (t)=I_{0^{+}}^{1-\gamma }\frac{d}{dt}\varphi
(t)=\frac{1}{\Gamma (1-\gamma )}\int_{0}^{t}(t-\tau )^{-\gamma }\varphi
^{\prime }(\tau )d\tau .
\end{equation*}
\end{definition}

\begin{definition}
\label{d2} \cite{Sa} (Sadik transform) Assume that $\varphi $ is piecewise
continuous on the interval $[0,A]$ for any $A$ $>0$ and satisfies $|\varphi
(t)|\leq Ke^{at}$ when $t$ $\geq M$, for any real constant $a$, and some
positive constants $K$ and $M$. Then the Sadik transform of $\varphi (t)$ is
defined by%
\begin{equation*}
\Phi (v,\alpha ,\beta )=\mathcal{S}[\varphi (t)]=\frac{1}{v^{\beta }}%
\int_{0}^{\infty }e^{-tv^{\alpha }}\varphi (t)dt,
\end{equation*}
\end{definition}

where $v$ is complex variable, $\alpha $ is any non zero real number, and $%
\beta $ is any real number.

\begin{definition}
\cite{IP} (Mittag-Leffer function) Let $p,q\in \mathcal{%
\mathbb{C}
}$, $Re(p)>0$, $Re(q)>0$. Then the Mittag-Leffer function of one variable is
given by%
\begin{equation*}
E_{p}(t)=\sum_{k=0}^{\infty }\frac{t^{k}}{\Gamma (pk+1)}.
\end{equation*}%
The Mittag-leffer function of two variables is given by
\begin{equation*}
E_{p,q}(t)=\sum_{k=0}^{\infty }\frac{t^{k}}{\Gamma (pk+q)}.
\end{equation*}
\end{definition}

\textbf{Properties: }\cite{Sa} Let $\Phi (v,\alpha ,\beta )$ be a Sadik
transform of $\varphi (t),$ i.e. $\mathcal{S}[\varphi (t)]=\Phi (v,\alpha
,\beta ).$ Then

\begin{enumerate}
\item If $\varphi (t)=1,$ then $\mathcal{S}[1]$ =$\frac{1}{v^{\alpha +\beta }%
}.$

\item If $\varphi (t)=t^{n},$ then $\mathcal{S}[t^{n}]=\frac{n!}{v^{n\alpha
+(\alpha +\beta )}}$.

\item If $\varphi (t)=e^{at}$, then $\mathcal{S}[e^{at}]=\frac{v^{-\beta }}{%
v^{\alpha }-a}.$

\item If $\varphi (t)=\sin (at),$ then $\mathcal{S}[\sin (at)]$=$\frac{%
av^{-\beta }}{v^{2\alpha }+a^{2}}.$
\end{enumerate}

\begin{lemma}
\cite{Sa} \label{t2} Let $\varphi _{1}$ and $\varphi _{2}$ two functions
belong to $L^{1}[a,b]$ the usuall convolution product is given by
\end{lemma}

\begin{equation*}
(\varphi _{1}\ast \varphi _{2})(t)=\int_{-\infty }^{\infty }\varphi
_{1}(\tau )\varphi _{2}(t-\tau )d\tau ,\text{ }t>0.
\end{equation*}

\begin{lemma}
\label{L2} \cite{Sa} Let $\Phi _{1}(v,\alpha ,\beta )$ and $\Phi
_{2}(v,\alpha ,\beta )$ are Sadik Transforms of $\varphi _{1}(t)$ and $%
\varphi _{2}(t)$ respectively, and $(\varphi _{1}\ast \varphi _{2})(t)$ is a
convolution \ of $\varphi _{1}(t)$ and $\varphi _{2}(t).$ Then, Sadik
transform of $(\varphi _{1}\ast \varphi _{2})(t)$\ is
\begin{equation*}
\mathcal{S[}(\varphi _{1}\ast \varphi _{2})(t)]=v^{\beta }\Phi _{1}(v,\alpha
,\beta ).\Phi _{2}(v,\alpha ,\beta ),
\end{equation*}%
where $\ast $ denotes convolution.
\end{lemma}

\begin{lemma}
\label{h} \cite{NS} Let $\mathcal{S}[\varphi (t)]=\Phi (v,\alpha ,\beta ).$
Then%
\begin{equation*}
\mathcal{S}[t^{n}\varphi (t)]=(-1)^{n}\left( \frac{1}{\alpha v^{\alpha -1}}%
\frac{d}{dv}+\frac{\beta }{\alpha v^{\alpha }}\right) ^{n}\Phi (v,\alpha
,\beta ).
\end{equation*}
\end{lemma}

\section{Main results}

In this section, we prove a new some properties of Sadik transform that is,
the Mittag-Liffler function, the integration, the time delay, the initial
and final value theorems. Moreover, we demonstrate the Sadik transform of
Caputo fractional differential equations, and the existence theorem of Sadik
transform.

\begin{theorem}
\label{L1} Let $\Phi (v,\alpha ,\beta )$ is a Sadik transform of $\varphi
(t) $ and $\varphi (t),$ $\varphi ^{\prime }(t),$ $\varphi ^{\prime \prime
}(t)$,$.... $,$\varphi ^{(n-1)}(t)$ are continuous on $[0,\infty)$. Then%
\begin{equation*}
\mathcal{S}[\varphi ^{(n)}(t)]=v^{n\alpha }\Phi (v,\alpha ,\beta
)-\sum_{k=0}^{n-1}v^{k\alpha -\beta }\varphi ^{(n-1-k)}(0).
\end{equation*}
\end{theorem}

\begin{proof}
For the first order derivative of $\varphi (t),$ we starting with the
definition of Sadik transform,%
\begin{equation*}
\mathcal{S}[\varphi ^{\prime }(t)]=v^{-\beta }\int_{0}^{\infty
}v^{-tv^{\alpha }}\varphi ^{\prime }(t)dt.
\end{equation*}%
By using integration of parts, we obtain,%
\begin{equation*}
=v^{-\beta }\left[ \left. v^{-tv^{\alpha }}\varphi (t)\right\vert
_{0}^{\infty }-\int_{0}^{\infty }v^{-tv^{\alpha }}(-v^{\alpha })\varphi (t)dt%
\right] ,
\end{equation*}%
Assuming $Re(v^{\alpha })>0$, we get%
\begin{eqnarray}
\mathcal{S}[\varphi ^{\prime }(t)] &=&-v^{-\beta }\varphi (0)+v^{\alpha }%
\frac{1}{v^{\beta }}\int_{0}^{\infty }v^{-tv^{\alpha }}\varphi (t)dt  \notag
\\
&=&v^{\alpha }\Phi (v,\alpha ,\beta )-v^{-\beta }\varphi (0).  \label{eq7}
\end{eqnarray}

By using a minner way of second order derivative of $\varphi (t)$ we get
\begin{eqnarray*}
\mathcal{S}[\varphi ^{\prime \prime }(t)] &=&v^{-\beta }\int_{0}^{\infty
}v^{-tv^{\alpha }}\varphi ^{\prime \prime }(t)dt \\
&=&v^{-\beta }\left[ \left. v^{-tv^{\alpha }}\varphi ^{\prime
}(t)\right\vert _{0}^{\infty }-\int_{0}^{\infty }v^{-tv^{\alpha
}}(-v^{\alpha })\varphi ^{\prime }(t)dt\right] \\
&=&-v^{-\beta }\varphi ^{\prime }(0)+v^{\alpha }v^{-\beta }\int_{0}^{\infty
}v^{-tv^{\alpha }}\varphi ^{\prime }(t)dt \\
&=&-v^{-\beta }\varphi ^{\prime }(0)+v^{\alpha }\mathcal{S}[\varphi ^{\prime
}(x)],
\end{eqnarray*}%
from Eq.(\ref{eq7}), we get%
\begin{eqnarray}
\mathcal{S}[\varphi ^{\prime \prime }(t)] &=&-v^{-\beta }\varphi ^{\prime
}(0)+v^{\alpha }\left[ v^{\alpha }U(v,\alpha ,\beta )-v^{-\beta }\varphi (0)%
\right]  \notag \\
&=&v^{2\alpha }U(v,\alpha ,\beta )-v^{\alpha -\beta }\varphi (0)-v^{-\beta
}\varphi ^{\prime }(0).  \label{eq6}
\end{eqnarray}

Similarly, for third order partial derivative of $\varphi (t)$ and using Eq.(%
\ref{eq6}), we get%
\begin{eqnarray*}
\mathcal{S}[\varphi ^{\prime \prime \prime }(t)] &=&v^{3\alpha }U(v,\alpha
,\beta )-v^{2\alpha -\beta }\varphi (0)-v^{\alpha -\beta }\varphi ^{\prime
}(0)-v^{-\beta }\varphi ^{\prime \prime }(0) \\
&=&v^{n\alpha }U(v,\alpha ,\beta )-\sum_{k=0}^{n-1}v^{k\alpha -\beta
}\varphi ^{(n-k-1)}(0),\text{ }n=3,k=0,1,2.
\end{eqnarray*}

Continually in the general, we get%
\begin{equation*}
\mathcal{S}[\varphi ^{(n)}(t)]=v^{n\alpha }U(v,\alpha ,\beta
)-\sum_{k=0}^{n-1}v^{k\alpha -\beta }\varphi ^{(n-k-1)}(0).
\end{equation*}
\end{proof}

\begin{lemma}
\label{L4} Let $\varphi (t)=t^{pm+q-1}E_{p,q}^{(m)}(\pm at^{p}).$ Then the
Sadik Transform of $\varphi (t)$ is given by%
\begin{equation*}
\frac{1}{v^{\beta }}\int_{0}^{\infty }e^{-v^{\alpha
}t}t^{pm+q-1}E_{p,q}^{(m)}(\pm at^{p})dt=\frac{m!v^{\alpha p-(\alpha q+\beta
)}}{(v^{\alpha p}\mp a)^{m+1}},
\end{equation*}%
where $\alpha ,\beta \in
\mathbb{C}
,\mathcal{R}e(p)>0,$ $\mathcal{R}e(q)>0,$ $\mathcal{R}e(v)>\left\vert
a\right\vert ^{\frac{1}{\mathcal{R}e(\alpha p)}}$ and $E_{p,q}^{(m)}(t)=%
\frac{d^{m}}{dt^{m}}E_{p,q}(t).$

\begin{proof}
In view of definitions of Mittag-leffler function and Sadik transform with
the help of classical calculus, we have%
\begin{eqnarray*}
&&\frac{1}{v^{\beta }}\int_{0}^{\infty }e^{-v^{\alpha
}t}t^{pm+q-1}E_{p,q}^{(m)}(\pm at^{p})dt \\
&=&\frac{1}{v^{\beta }}\int_{0}^{\infty }e^{-v^{\alpha }t}t^{pm+q-1}\frac{%
d^{m}}{dt^{m}}\sum_{k=0}^{\infty }\frac{(\pm at^{p})^{k}}{\Gamma (pk+q)}dt \\
&=&\frac{1}{v^{\beta }}\int_{0}^{\infty }e^{-v^{\alpha
}t}t^{pm+q-1}\sum_{k=0}^{\infty }\frac{(k+m)!(\pm a)^{k}t^{pk}}{k!\Gamma
(pk+pm+q)}dt \\
&=&\sum_{k=0}^{\infty }\frac{(k+m)!(\pm a)^{k}}{k!\Gamma (p\left( k+m\right)
+q)}\frac{1}{v^{\beta }}\int_{0}^{\infty }e^{-v^{\alpha }t}t^{p\left(
m+k\right) +q-1}dt \\
&=&\frac{v^{-\alpha (pm+q)}}{v^{\beta }}\sum_{k=0}^{\infty }\frac{(k+m)!}{k!}%
\left( \frac{\pm a}{v^{\alpha p}}\right) ^{k} \\
&=&\frac{v^{-\alpha (pm+q)}}{v^{\beta }}\sum_{k=0}^{\infty
}(k+m).......(k+1)\left( \frac{\pm a}{v^{\alpha p}}\right) ^{k}.
\end{eqnarray*}%
Now let $k=k-m$, it follows that%
\begin{eqnarray}
&&\frac{1}{v^{\beta }}\int_{0}^{\infty }e^{-v^{\alpha
}t}t^{pm+q-1}E_{p,q}^{(m)}(\pm at^{p})dt  \notag \\
&=&\frac{v^{-\alpha (pm+q)}}{v^{\beta }}\sum_{k=m}^{\infty
}(k)(k-1)...(k-m-1)\left( \frac{\pm a}{v^{\alpha p}}\right) ^{k}  \notag \\
&=&v^{-\alpha pm-\alpha q-\beta }\frac{d^{m}}{da^{m}}\sum_{k=m}^{\infty
}\left( \frac{\pm a}{v^{\alpha p}}\right) ^{k}  \notag \\
&=&v^{-\alpha pm-\alpha q-\beta }\frac{d^{m}}{da^{m}}\sum_{k=m}^{\infty
}\left( \frac{1}{1\mp \frac{a}{v^{\alpha p}}}\right)  \notag \\
&=&v^{-\alpha pm-\alpha q-\beta }\frac{m!}{\left( 1\mp \frac{a}{v^{\alpha p}}%
\right) ^{m+1}}  \notag \\
&=&\frac{m!v^{\alpha p-(\alpha q+\beta )}}{(v^{\alpha p}\mp a)^{m+1}}.
\label{eq23}
\end{eqnarray}
\end{proof}
\end{lemma}

\begin{lemma}
(Integration) Let $\mathcal{S}[\varphi (t)]=\Phi (v,\alpha ,\beta )$ is a
Sadik transform of $\varphi (t).$ Then Sadik transform of integration of $%
\varphi (t)$ is
\begin{equation*}
\mathcal{S}\left[ \int_{0}^{t}\varphi (\tau )d\tau \right] =\frac{1}{%
v^{\alpha }}\Phi (v,\alpha ,\beta ).
\end{equation*}
\end{lemma}

\begin{proof}
According to the definition of Sadik transform, we have%
\begin{equation*}
\mathcal{S}\left[ \int_{0}^{t}\varphi (\tau )d\tau \right] =\frac{1}{%
v^{\beta }}\int_{0}^{\infty }e^{-tv^{\alpha }}\left[ \int_{0}^{t}\varphi
(\tau )d\tau \right] dt.
\end{equation*}

Then by the integration of parts, we get%
\begin{eqnarray*}
\mathcal{S}\left[ \int_{0}^{t}\varphi (\tau )d\tau \right] &=&\frac{1}{%
v^{\beta }}\left[ \left. \int_{0}^{t}\varphi (\tau )d\tau \frac{%
e^{-tv^{\alpha }}}{v^{\alpha }}\right\vert _{0}^{\infty }-\int_{0}^{\infty }%
\frac{e^{-tv^{\alpha }}}{-v^{\alpha }}\right] \varphi (t)dt \\
&=&\frac{1}{v^{\alpha }}\left[ \frac{1}{v^{\beta }}\int_{0}^{\infty
}e^{-tv^{\alpha }}\varphi (t)dt\right] \\
&=&\frac{1}{v^{\alpha }}\Phi (v,\alpha ,\beta ).
\end{eqnarray*}
\end{proof}

\begin{theorem}
\label{th2} Let $n-1<\gamma <n,$ $\left( n=[\gamma ]+1\right) $ and $\varphi
(t),\varphi ^{\prime }(t)$,$\varphi ^{\prime \prime }(t),.....,\varphi
^{(n-1)}(t)$ are continuous on $[0,\infty )$ and of exponential order, while
$^{c}D_{0^{+}}^{\gamma }\varphi (t)$ is piecewise continuous on $[0,\infty
). $ Then Sadik transform of Caputo fractional derivative of order $\gamma $
of function $\varphi $ is given by%
\begin{equation*}
\mathcal{S}[^{c}D_{0^{+}}^{\gamma }\varphi (t)]=v^{\gamma \alpha }\Phi
(v,\alpha ,\beta )-\sum_{k=0}^{n-1}v^{(\gamma -n+k)\alpha -\beta }\varphi
^{(n-1-k)}(0^{+}).
\end{equation*}
\end{theorem}

\begin{proof}
In light of Definitions \ref{d1}, \ref{d2}, then using Theorem \ref{L1}, and
property 2, we find that
\begin{eqnarray*}
\mathcal{S}[^{c}D_{0^{+}}^{\gamma }\varphi (t)] &=&\frac{1}{v^{\beta }}%
\int_{0}^{\infty }e^{-tv^{\alpha }}\left[ ^{c}D_{0^{+}}^{\gamma }\varphi (t)%
\right] dt \\
&=&\frac{1}{v^{\beta }}\int_{0}^{\infty }e^{-tv^{\alpha }}\left[ \frac{1}{%
\Gamma (n-\gamma )}\int_{0}^{t}(t-\tau )^{n-\gamma -1}\varphi ^{(n)}(\tau
)d\tau \right] dt \\
&=&\frac{1}{\Gamma (n-\gamma )}\frac{1}{v^{\beta }}\int_{0}^{\infty
}\int_{\tau }^{\infty }e^{-tv^{\alpha }}(t-\tau )^{n-\gamma -1}\varphi
^{(n)}(\tau )dtd\tau \\
&=&\frac{1}{\Gamma (n-\gamma )}\frac{1}{v^{\beta }}\int_{0}^{\infty }\varphi
^{(n)}(\tau )\int_{0}^{\infty }e^{-v^{\alpha }(u+\tau )}u^{n-\gamma
-1}dud\tau \\
&=&\frac{1}{\Gamma (n-\gamma )}\frac{1}{v^{\beta }}\int_{0}^{\infty
}e^{-\tau v^{\alpha }}\varphi ^{(n)}(\tau )d\tau \int_{0}^{\infty
}e^{-uv^{\alpha }}u^{n-\gamma -1}du \\
&=&\frac{1}{\Gamma (n-\gamma )}\int_{0}^{\infty }e^{-\tau v^{\alpha
}}\varphi ^{(n)}(\tau )d\tau \frac{1}{v^{\beta }}\int_{0}^{\infty
}e^{-uv^{\alpha }}u^{n-\gamma -1}du \\
&=&\frac{v^{\beta }}{\Gamma (n-\gamma )}\mathcal{S}\left[ \varphi ^{(n)}(t)%
\right] \mathcal{S}\left[ t^{n-\gamma -1}\right] \\
&=&\frac{v^{\beta }}{\Gamma (n-\gamma )}\mathcal{S}\left[ \varphi ^{(n)}(t)%
\right] \frac{\Gamma (n-\gamma )}{v^{(n-\gamma -1)\alpha +(\alpha +\beta )}}
\\
&=&\frac{v^{\beta }}{v^{(n-\gamma )\alpha +\beta }}\left[ v^{n\alpha }\Phi
(v,\alpha ,\beta )-\sum_{k=0}^{n-1}v^{k\alpha -\beta }\varphi ^{(n-1-k)}(0)%
\right] \\
&=&v^{\alpha \gamma }\Phi (v,\alpha ,\beta )-\sum_{k=0}^{n-1}v^{(\gamma
+k-n)\alpha -\beta }\varphi ^{(n-1-k)}(0).
\end{eqnarray*}
\end{proof}

\begin{theorem}
\label{th1} Assume that a linear Caputo fractional differential equation
\begin{equation}
^{c}D_{0^{+}}^{\gamma }u(t)=\varphi (t),\text{ \ }0<\gamma <1,  \label{eq1}
\end{equation}%
with intial condition\
\begin{equation}
u(0)=u_{0},\text{\qquad \qquad \qquad\ \ \ \ \ }  \label{eq2}
\end{equation}%
has a unique continuous solution
\begin{equation}
u(t)=u_{0}+\frac{1}{\Gamma (\gamma )}\int_{0}^{t}(t-\tau )^{\gamma
-1}\varphi (\tau )d\tau ,  \label{eq3}
\end{equation}%
if $\varphi (t)$ is continuous on $[0,\infty )$ and exponentially bounded,
then $u(t)$ and $^{c}D_{0^{+}}^{\gamma }u(t)$ are both exponentially
bounded, thus their Sadik transform exists.
\end{theorem}

\begin{proof}
Since $\varphi (t)$ is exponentially bounded, there exist two positive
constants $M,\sigma $ and enough large $T$ such that $\left\Vert \varphi
(t)\right\Vert \leq Me^{\sigma t}$ for all $t\geq T$. It is easy to see that
Eq.(\ref{eq1}) is equivalent to the Volterra integral equation%
\begin{equation}
u(t)=u_{0}+\frac{1}{\Gamma (\gamma )}\int_{0}^{t}(t-\tau )^{\gamma
-1}\varphi (\tau )d\tau \text{, \ }0\leq t<\infty .  \label{eq4}
\end{equation}

For $t\geq T,$ Eq.(\ref{eq4}) can be rewritten as%
\begin{equation*}
u(t)=u_{0}+\frac{1}{\Gamma (\gamma )}\int_{0}^{T}(t-\tau )^{\gamma
-1}\varphi (\tau )d\tau +\frac{1}{\Gamma (\gamma )}\int_{T}^{t}(t-\tau
)^{\gamma -1}\varphi (\tau )d\tau .
\end{equation*}

In view of assumptions, $u(t)$ is unique continuous solution on $[0,\infty )$%
, with $u(0)=u_{0},$ then $\varphi (t)$ is bounded on $[0,T]$, i.e. there
exists a constant $k>0$ such that $\Vert \varphi (t)\Vert \leq k$. Now, we
have%
\begin{equation*}
\left\Vert u(t)\right\Vert \leq \left\Vert u_{0}\right\Vert +\frac{k}{\Gamma
(\gamma )}\int_{0}^{T}(t-\tau )^{\gamma -1}d\tau +\frac{1}{\Gamma (\gamma )}%
\int_{T}^{t}(t-\tau )^{\gamma -1}\left\Vert \varphi (\tau )\right\Vert d\tau
.
\end{equation*}

Multiply the last inequality by $e^{-\sigma t}$ then from fact that $%
e^{-\sigma t}\leq e^{-\sigma T}$, $e^{-\sigma t}$ $\leq $ $e^{-\sigma \tau }$%
, and $\Vert \varphi (t)\Vert \leq Me^{\sigma t}$ $(t\geq T)$, we obtain%
\begin{eqnarray*}
\left\Vert u(t)\right\Vert e^{-\sigma t} &\leq &\left\Vert u_{0}\right\Vert
e^{-\sigma t}+\frac{ke^{-\sigma t}}{\Gamma (\gamma )}\int_{0}^{T}(t-\tau
)^{\gamma -1}d\tau +\frac{e^{-\sigma t}}{\Gamma (\gamma )}%
\int_{T}^{t}(t-\tau )^{\gamma -1}\left\Vert \varphi (\tau )\right\Vert d\tau
\\
&\leq &\left\Vert u_{0}\right\Vert e^{-\sigma T}+\frac{ke^{-\sigma T}}{%
\Gamma (\gamma +1)}[(t)^{\gamma }-(t-T)^{\gamma }]+\frac{M}{\Gamma (\gamma )}%
\int_{0}^{t}(t-\tau )^{\gamma -1}e^{\sigma (\tau -t)}d\tau \\
&\leq &\left\Vert u_{0}\right\Vert e^{-\sigma T}+\frac{ke^{-\sigma T}}{%
\Gamma (\gamma +1)}T^{\gamma }+\frac{M}{\Gamma (\gamma )}\int_{0}^{t}s^{%
\gamma -1}e^{-\sigma s}ds \\
&\leq &\left\Vert u_{0}\right\Vert e^{-\sigma T}+\frac{ke^{-\sigma T}}{%
\Gamma (\gamma +1)}T^{\gamma }+\frac{M}{\Gamma (\gamma )}\int_{0}^{\infty
}s^{\gamma -1}e^{-\sigma s}ds \\
&\leq &\left\Vert u_{0}\right\Vert e^{-\sigma T}+\frac{ke^{-\sigma T}}{%
\Gamma (\gamma +1)}T^{\gamma }+\frac{M}{\sigma ^{\gamma }}.
\end{eqnarray*}

Denote
\begin{equation*}
A=\left\Vert u_{0}\right\Vert e^{-\sigma T}+\frac{ke^{-\sigma T}}{\Gamma
(\gamma +1)}T^{\gamma }+\frac{M}{\sigma ^{\gamma }},
\end{equation*}%
we get%
\begin{equation*}
\left\Vert u(t)\right\Vert \leq Ae^{\sigma t},\text{ }t\geq T.
\end{equation*}
\end{proof}

From Eq.(\ref{eq1}) and hypothesis of $\varphi ,$ we conclude that%
\begin{equation*}
\left\Vert ^{c}D_{0^{+}}^{\gamma }u(t)\right\Vert =\left\Vert \varphi
(t)\right\Vert \leq Me^{\sigma t}\text{ }t\geq T.
\end{equation*}

Applying Sadik transform on both sides of Eq.(\ref{eq1}) and using Theorem %
\ref{th2}, we have
\begin{equation*}
v^{\alpha \gamma }U(v,\alpha ,\beta )-v^{(\gamma -1)\alpha -\beta }u(0)=\Phi
(v,\alpha ,\beta ).
\end{equation*}

Since $u(0)=u_{0},$ it follows
\begin{equation*}
U(v,\alpha ,\beta )=u_{0}\frac{1}{v^{\alpha +\beta }}+\frac{\Phi (v,\alpha
,\beta )}{v^{\alpha \gamma }}.
\end{equation*}

Take the inverse of Sadik transform to both sides of the above equation, and
using property 1, and Lemma \ref{L2}, we get%
\begin{eqnarray}
u(t) &=&u_{0}\mathcal{S}^{-1}\left[ \frac{1}{v^{\alpha +\beta }}\right] +%
\mathcal{S}^{-1}\left[ \frac{1}{v^{\alpha \gamma }}\Phi (v,\alpha ,\beta )%
\right]  \notag \\
&=&u_{0}+\mathcal{S}^{-1}\left[ v^{\beta }\frac{1}{v^{\alpha \gamma +\beta }}%
\Phi (v,\alpha ,\beta )\right]  \notag \\
&=&u_{0}+\mathcal{S}^{-1}\left[ v^{\beta }\frac{1}{v^{(\gamma -1)\alpha
+\alpha +\beta }}\Phi (v,\alpha ,\beta )\right]  \notag \\
&=&u_{0}+(\varphi _{1}\ast \varphi )(t).  \label{eq9}
\end{eqnarray}

Put $\Phi _{1}(v,\alpha ,\beta ):=\frac{1}{v^{(\gamma -1)\alpha +\alpha
+\beta }},$ such that $\mathcal{S}^{-1}\left[ \Phi _{1}(v,\alpha ,\beta )%
\right] =\varphi _{1}(t)$ and $\mathcal{S}^{-1}[\Phi (v,\alpha ,\beta
)]=\varphi (t).$ Applying the inverse Sadik transform of $\Phi _{1}(v,\alpha
,\beta )$, with using property 2, we find that
\begin{equation*}
\mathcal{S}^{-1}\left[ \Phi _{1}(v,\alpha ,\beta )\right] =\mathcal{S}^{-1}%
\left[ \frac{1}{v^{(\gamma -1)\alpha +\alpha +\beta }}\right] =\frac{%
t^{\gamma -1}}{\Gamma (\gamma )}=\varphi _{1}(t).
\end{equation*}

Therefore Eq.(\ref{eq9}) becomes as follows
\begin{eqnarray*}
u(t) &=&u_{0}+(\varphi _{1}\ast \varphi )(t) \\
&=&u_{0}+\frac{1}{\Gamma (\gamma )}\int_{0}^{t}(t-\tau )^{\gamma -1}\varphi
(\tau )d\tau .
\end{eqnarray*}

\begin{theorem}
\label{th3}(Time Delay) Let $\Phi (v,\alpha ,\beta )=\mathcal{S[}\varphi
(t)] $. Then Sadik transform of time delay is given by
\begin{equation*}
\mathcal{S}\left[ \varphi (t-a)\cdot \eta (t-a)\right] =e^{-av^{\alpha
}}\Phi (v,\alpha ,\beta ),
\end{equation*}%
where
\begin{equation}
\eta (t-a)=\left\{
\begin{array}{c}
0,\qquad t<a \\
1,\qquad t\geq a%
\end{array}%
\right. .  \label{eq5}
\end{equation}
\end{theorem}

\begin{proof}
We prove by going back to the original definition of the Sadik transform%
\begin{equation*}
\mathcal{S}\left[ \varphi (t-a)\cdot \eta (t-a)\right] =\frac{1}{v^{\beta }}%
\int_{0}^{\infty }e^{-tv^{\alpha }}\left[ \varphi (t-a)\cdot \eta (t-a)%
\right] dt,
\end{equation*}

Changing the lower limit of the integral from $0$ to $a$ and drop the step
function gives%
\begin{eqnarray*}
\mathcal{S}\left[ \varphi (t-a)\cdot \eta (t-a)\right] &=&\frac{1}{v^{\beta }%
}\int_{0}^{a}e^{-tv^{\alpha }}\left[ \varphi (t-a)\cdot \eta (t-a)\right] dt
\\
&&+\frac{1}{v^{\beta }}\int_{a}^{\infty }e^{-tv^{\alpha }}\left[ \varphi
(t-a)\cdot \eta (t-a)\right] dt \\
&=&\frac{1}{v^{\beta }}\int_{a}^{\infty }e^{-tv^{\alpha }}\varphi (t-a)dt.
\end{eqnarray*}

By change of variable $u=t-a,$ it follows%
\begin{eqnarray*}
\mathcal{S}\left[ \varphi (t-a)\cdot \eta (t-a)\right] &=&\frac{1}{v^{\beta }%
}\int_{0}^{\infty }e^{-(u+a)v^{\alpha }}\varphi (u)du \\
&=&e^{-av^{\alpha }}\Phi (v,\alpha ,\beta ).
\end{eqnarray*}
\end{proof}

\begin{example}
\label{u} The Sadik transform of $\eta (t-a)$ is $\frac{e^{-av^{\alpha }}}{%
v^{\alpha +\beta }}$. Indeed, from the definition of Sadik transform and the
relation Eq.(\ref{eq5}), we have%
\begin{equation*}
\mathcal{S[}\eta (t-a)]=\frac{1}{v^{\beta }}\int_{0}^{a}e^{-tv^{\alpha
}}\eta (t-a)dt+\frac{1}{v^{\beta }}\int_{a}^{\infty }e^{-tv^{\alpha }}\eta
(t-a)dt=\frac{e^{-av^{\alpha }}}{v^{\alpha +\beta }}.
\end{equation*}%
Setting $\varphi (t-a)=1.$ Then
\begin{equation*}
\mathcal{S}\left[ \varphi (t-a)\cdot \eta (t-a)\right] =\mathcal{S}\left[
\eta (t-a)\right] =\frac{e^{-av^{\alpha }}}{v^{\alpha +\beta }}%
=e^{-av^{\alpha }}\Phi (v,\alpha ,\beta ).
\end{equation*}%
This satisfies Theorem \ref{th3}.
\end{example}

\begin{example}
Let $\varphi (t-a)=t-a.$ Then, from Lemma \ref{h}, we have
\begin{equation*}
\mathcal{S}[t\varphi (t)]=-\left( \frac{1}{\alpha v^{\alpha -1}}\frac{d}{dv}+%
\frac{\beta }{\alpha v^{\alpha }}\right) \Phi (v,\alpha ,\beta ).
\end{equation*}%
Hence, with using Example \ref{u}, we get
\begin{eqnarray*}
&&\mathcal{S}\left[ \varphi (t-a).\eta (t-a)\right] \\
&=&-\left( \frac{1}{\alpha v^{\alpha -1}}\frac{d}{dv}+\frac{\beta }{\alpha
v^{\alpha }}\right) \frac{e^{-av^{\alpha }}}{v^{\alpha +\beta }}-a\frac{%
e^{-av^{\alpha }}}{v^{\alpha +\beta }} \\
&=&-\frac{1}{\alpha v^{\alpha -1}}\left[ \frac{(-a\alpha v^{\alpha
-1}-\left( \alpha +\beta \right) v^{-1})e^{-av^{\alpha }}}{v^{\alpha +\beta }%
}\right] +\frac{\beta }{\alpha }\frac{e^{-av^{\alpha }}}{v^{2\alpha +\beta }}%
-a\frac{e^{-av^{\alpha }}}{v^{\alpha +\beta }} \\
&=&\frac{1}{v^{2\alpha +\beta }}e^{-av^{\alpha }} \\
&=&e^{-av^{\alpha }}\Phi (v,\alpha ,\beta ).
\end{eqnarray*}
\end{example}

\begin{theorem}
\label{th4} (Initial Value Theorem) Let $\Phi (v,\alpha ,\beta )=\mathcal{S[}%
\varphi (t)]$. Then the sadik transform of initial value given by%
\begin{equation*}
\lim_{v^{\alpha }\rightarrow \infty }\left[ v^{\alpha }\Phi (v,\alpha ,\beta
)\right] =v^{-\beta }\varphi (0^{+}).
\end{equation*}
\end{theorem}

\begin{proof}
We first start with the derivative rule:%
\begin{equation*}
\mathcal{S}\left[ \frac{d\varphi (t)}{dt}\right] =v^{\alpha }\Phi (v,\alpha
,\beta )-v^{-\beta }\varphi (0^{-}).
\end{equation*}

From the definition of Sadik transform, with splitting the integral into two
parts:%
\begin{eqnarray*}
v^{\alpha }\Phi (v,\alpha ,\beta )-v^{-\beta }\varphi (0^{-}) &=&\frac{1}{%
v^{\beta }}\int_{0^{-}}^{\infty }\frac{d\varphi (t)}{dt}e^{-tv^{\alpha }}dt=%
\frac{1}{v^{\beta }}\int_{0^{-}}^{\infty }\varphi ^{\prime
}(t)e^{-tv^{\alpha }}dt \\
&=&\frac{1}{v^{\beta }}\int_{0^{-}}^{0^{+}}\varphi ^{\prime
}(t)e^{-tv^{\alpha }}dt+\frac{1}{v^{\beta }}\int_{0^{+}}^{\infty }\varphi
^{\prime }(t)e^{-tv^{\alpha }}dt.
\end{eqnarray*}

Take the limit as $v^{\alpha }\rightarrow \infty ,$%
\begin{eqnarray*}
&&\lim_{v^{\alpha }\rightarrow \infty }\left[ v^{\alpha }\Phi (v,\alpha
,\beta )-v^{-\beta }\varphi (0^{-})\right]  \\
&=&\lim_{v^{\alpha }\rightarrow \infty }\left[ \frac{1}{v^{\beta }}%
\int_{0^{-}}^{0^{+}}\varphi ^{\prime }(t)e^{-tv^{\alpha }}dt+\frac{1}{%
v^{\beta }}\int_{0^{+}}^{\infty }\varphi ^{\prime }(t)e^{-tv^{\alpha }}dt%
\right] .
\end{eqnarray*}

Several facilitations are as follows:

In the expression
\begin{equation*}
\lim_{v^{\alpha }\rightarrow \infty }\left[ v^{\alpha }\Phi (v,\alpha ,\beta
)-v^{-\beta }\varphi (0^{-})\right] ,
\end{equation*}%
we can take the second term out of the limit, since it doesn't depend on $%
v^{\alpha }.$

In the expression
\begin{equation*}
\lim_{v^{\alpha }\rightarrow \infty }\left[ \frac{1}{v^{\beta }}%
\int_{0^{-}}^{0^{+}}\varphi ^{\prime }(t)e^{-tv^{\alpha }}dt+\frac{1}{%
v^{\beta }}\int_{0^{+}}^{\infty }\varphi ^{\prime }(t)e^{-tv^{\alpha }}dt%
\right] ,
\end{equation*}%
we can take the first term out of the limit for the same reason, and when $%
v^{\alpha }\rightarrow \infty $ the exponential term in the second term goes
to zero. Hence%
\begin{eqnarray*}
\left( \lim_{v^{\alpha }\rightarrow \infty }[v^{\alpha }\Phi (v,\alpha
,\beta )]\right) -v^{-\beta }\varphi (0^{-}) &=&\frac{1}{v^{\beta }}%
\int_{0^{-}}^{0^{+}}\varphi ^{\prime }(t)dt+\frac{1}{v^{\beta }}%
\int_{0^{+}}^{\infty }\varphi ^{\prime }(t)(0)dt \\
&=&\frac{1}{v^{\beta }}\int_{0^{-}}^{0^{+}}\varphi ^{\prime }(t)dt \\
&=&v^{-\beta }\varphi (0^{+})-v^{-\beta }\varphi (0^{-}).
\end{eqnarray*}

This gives%
\begin{equation*}
\lim_{v^{\alpha }\rightarrow \infty }\left[ v^{\alpha }\Phi (v,\alpha ,\beta
)\right] =v^{-\beta }\varphi (0^{+}).
\end{equation*}
\end{proof}

\begin{remark}
Theorem \ref{th4} is true only if $\varphi (t)$ is a strictly proper
fraction in which the numerator order is lower than the denominator order.
\end{remark}

\begin{theorem}
\label{th5} (Final Value Theorem) Let $\Phi (v,\alpha ,\beta )=\mathcal{S[}%
\varphi (t)]$. Then the sadik transform of final value given by
\begin{equation*}
\lim_{v^{\alpha }\rightarrow 0}\left[ v^{\alpha }\Phi (v,\alpha ,\beta )%
\right] =\lim_{t\rightarrow \infty }\left[ v^{-\beta }\varphi (t)\right] .
\end{equation*}

\begin{proof}
We start as we did for the initial value theorem, with the Sadik transform
of the derivative%
\begin{equation*}
\mathcal{S}\left[ \frac{d\varphi (t)}{dt}\right] =\frac{1}{v^{\beta }}%
\int_{0^{-}}^{\infty }\frac{d\varphi (t)}{dt}e^{-tv^{\alpha }}dt=v^{\alpha
}\Phi (v,\alpha ,\beta )-v^{-\beta }\varphi (0^{-}).
\end{equation*}

Take the limit on both sides as $v^{\alpha }\rightarrow 0,$ we have%
\begin{equation*}
\lim_{v^{\alpha }\rightarrow 0}\left[ \frac{1}{v^{\beta }}%
\int_{0^{-}}^{\infty }\frac{d\varphi (t)}{dt}e^{-tv^{\alpha }}dt\right]
=\lim_{v^{\alpha }\rightarrow 0}\left[ v^{\alpha }\Phi (v,\alpha ,\beta
)-v^{-\beta }\varphi (0^{-})\right] .
\end{equation*}

As $v^{\alpha }\rightarrow 0,$ $e^{-tv^{\alpha }}$vanishes from the
integral. Also, the term $v^{-\beta }\varphi (0^{-})$ in the right side we
can take it out of the limit since it independent of $v^{\alpha }.$ Hence,
by the theory of fundamental calculus, we have%
\begin{eqnarray*}
\lim_{v^{\alpha }\rightarrow 0}\left( \frac{1}{v^{\beta }}\varphi (\infty )-%
\frac{1}{v^{\beta }}\varphi (0^{-})\right) &=&\lim_{v^{\alpha }\rightarrow 0}%
\left[ \frac{1}{v^{\beta }}\int_{0^{-}}^{\infty }\frac{d\varphi (t)}{dt}dt%
\right] \\
&=&\left( \lim_{v^{\alpha }\rightarrow 0}\left[ v^{\alpha }\Phi (v,\alpha
,\beta )\right] \right) -v^{-\beta }\varphi (0^{-}).
\end{eqnarray*}

Since the term on the left doesn't depend on $v^{\alpha }$, thus%
\begin{equation*}
\frac{1}{v^{\beta }}\varphi (\infty )-\frac{1}{v^{\beta }}\varphi
(0^{-})=\left( \lim_{v^{\alpha }\rightarrow 0}\left[ v^{\alpha }\Phi
(v,\alpha ,\beta )\right] \right) -v^{-\beta }\varphi (0^{-}).
\end{equation*}

That is
\begin{equation*}
\lim_{t\rightarrow \infty }\left[ v^{-\beta }\varphi (t)\right]
=\lim_{v^{\alpha }\rightarrow 0}\left[ v^{\alpha }\Phi (v,\alpha ,\beta )%
\right] .
\end{equation*}
\end{proof}
\end{theorem}

\begin{remark}
Theorem \ref{th5} is satisfied for all functions except the increasing
functions and oscillating functions such as sine and cosine that don't have
a final value.
\end{remark}

\begin{example}
\label{rw} Let $0<\gamma <1$ and $b\in
\mathbb{R}
.$ Then the problem
\begin{equation}
^{c}D_{0}^{\gamma }y(t)-by(t)=0  \label{eq8}
\end{equation}%
with the initial condition $y(0)=y_{0}$ has a solution given by
\begin{equation*}
y(t)=y_{0}\sum_{k=0}^{\infty }\frac{(bt^{\gamma })^{k}}{\Gamma (\gamma k+1)}%
=y_{0}E_{\gamma ,1}(bt^{\gamma }).
\end{equation*}%
Applying the Sadik transform on both sides of Eq.(\ref{eq8}), together with
the Theorem \ref{th2}, we can conclude that%
\begin{equation*}
Y(v,\alpha ,\beta )=\frac{v^{\alpha \gamma -\beta -1}}{(v^{\alpha \gamma }-b)%
}y_{0},
\end{equation*}%
by using the lemma \ref{L4} we get
\begin{equation*}
y(t)=y_{0}\sum_{k=0}^{\infty }\frac{(bt^{\gamma })^{k}}{\Gamma (\gamma k+1)}%
=y_{0}E_{\gamma ,1}(bt^{\gamma })
\end{equation*}
\end{example}

\begin{table}[tbp]
\caption{ Plot of y(t) versus t (h=0.2) in Example \protect\ref{rw}}
\begin{center}
\begin{tabular}{|l|l||l|l||l|l||l|}
\hline
t & 0 & 0.2000 & 0.4000 & 0.6000 & 0.8000 & 1.0000 \\ \hline
y(t) & -1.1284 & -0.2012 & -0.0649 & 0.0479 & 0.1743 & 0 \\ \hline
\end{tabular}%
\end{center}
\end{table}

\begin{example}
Let $\varphi (t)=e^{t}.$ Then $\Phi (v,\alpha ,\beta )=\frac{v^{-\beta }}{%
v^{\alpha }-1}$ and the initial value of this function given as follows%
\begin{equation*}
\lim_{v^{\alpha }\rightarrow \infty }\left[ v^{\alpha }\varphi (v,\alpha
,\beta ,)\right] =\lim_{v^{\alpha }\rightarrow \infty }\left[ v^{\alpha }%
\frac{v^{-\beta }}{v^{\alpha }-1}\right] =v^{-\beta }=v^{-\beta }\varphi
(0^{+}),
\end{equation*}

where $\varphi (0^{+})=1.$ Therefore, Theorem (\ref{th4}) is satisfied.
\end{example}

\begin{figure}[tbp]
\includegraphics[width=\linewidth]{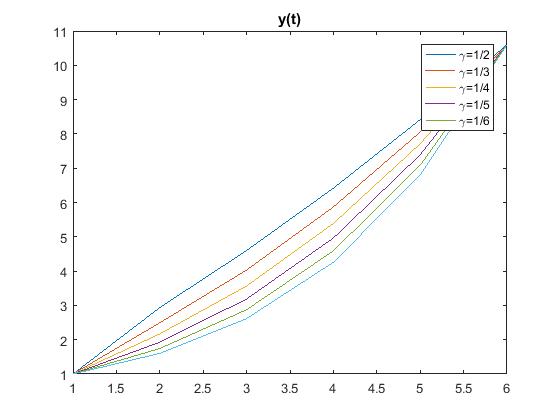}
\caption{Plot of $y(t)=y_{0}E_{\protect\gamma ,1}(bt^{\protect\gamma })$ for
the Caputo fractional problem and the initial condition $y(0)=1$, $b=3$ with
diferent values of $\protect\gamma $ in Example \protect\ref{rw}.}
\label{fig:boat1}
\end{figure}

\begin{example}
Let $\varphi (t)=\sin (at).$ Then $\Phi (v,\alpha ,\beta ,)=\frac{av^{-\beta
}}{v^{2\alpha }+a^{2}},$ and the initial value of this function given as
follows%
\begin{equation*}
\lim_{v^{\alpha }\rightarrow \infty }\left[ v^{\alpha }\varphi (v,\alpha
,\beta ,)\right] =\lim_{v^{\alpha }\rightarrow \infty }\left[ v^{\alpha }%
\frac{av^{-\beta }}{v^{2\alpha }+a^{2}}\right] =0=v^{-\beta }\varphi (0^{+}),
\end{equation*}

where $\varphi (0^{+})=0.$ This means that Theorem (\ref{th4}) holds.
\end{example}

\begin{example}
If we have $\varphi (t)=1,$ then $\Phi (v,\alpha ,\beta ,)=\frac{1}{%
v^{\alpha +\beta }}$ and%
\begin{equation*}
\lim_{v^{\alpha }\rightarrow 0}\left[ v^{\alpha }\varphi (v,\alpha ,\beta ,)%
\right] =\lim_{v^{\alpha }\rightarrow 0}\left[ \frac{v^{\alpha }}{v^{\alpha
+\beta }}\right] =v^{-\beta }=v^{-\beta }\lim_{t\rightarrow \infty }\varphi
(t),
\end{equation*}

where $\lim_{t\rightarrow \infty }\varphi (t)=1.$ By using the Theorem (\ref%
{th5}), we get the final value of $\varphi (t)$ that is $v^{-\beta }.$
\end{example}

\begin{example}
Consider the Dirac delta function
\begin{equation*}
\delta (t)=\left\{
\begin{array}{c}
\infty ,\qquad t=0, \\
0,\qquad t\neq 0.%
\end{array}%
\right.
\end{equation*}%
\

Then $\Phi (v,\alpha ,\beta ,)=\mathcal{S[\delta (}t\mathcal{)]=}1.$ In view
of Theorem (\ref{th5}), the final value of this function is
\begin{equation*}
\lim_{v^{\alpha }\rightarrow 0}\left[ v^{\alpha }\varphi (v,\alpha ,\beta ,)%
\right] =\lim_{v^{\alpha }\rightarrow 0}\left[ v^{\alpha }\right] =0
\end{equation*}

and
\begin{equation*}
\lim_{t\rightarrow \infty }v^{-\beta }\delta (t)=v^{-\beta
}\lim_{t\rightarrow \infty }\left\{
\begin{array}{c}
\infty ,\qquad t=0 \\
0,\qquad t\neq 0%
\end{array}%
\right. =0.
\end{equation*}

So our results are satisfied.
\end{example}

\section{Application}

\subsection{Fractional-order dynamical systems in control theory.}

Modern and effective techniques for the time-domain analysis of dynamical
systems involving fractional-order are wanted to solving systems of control
theory. As a modern generalization of the ordinary $PID$-controller, the
idea of $PI^{\lambda }D^{%
{\mu}%
}$-controller, including fractional-order integrator and fractional-order
differentiator, has been lead to be a more efficient control dynamical
systems of fractional-order. In his series of papers and books (see
references of Podlubny's book \cite{IP}), successfully applied the
fractional-order controller to improve the so-called CRONE-controller
(Commande Robuste d'Ordre Non-En-trier controller) which is an enjoyable
example of the application of fractional derivatives in control theory. He
proves the advantage of the CRONE-controller compared to the classical
PID-controller and also showed that the $PI^{\lambda }D^{%
{\mu}%
}$-controller has a better rendering record when applied for the control of
fractional-order systems than the classical $PID$-controller. In the time
domain, he described a dynamical system by the fractional-order differential
equation (FDE)%
\begin{equation}
\left[ \sum_{k=0}^{n}r_{n-k}\text{ }^{c}D_{0^{+}}^{\gamma _{n-k}}\right]
\varphi (t)=f(t),  \label{12}
\end{equation}%
where $\gamma _{n-k}>\gamma _{n-k-1\text{ }},(k=0,1,2,...,n)$ are arbitrary
real numbers, $r_{n-k}$ are arbitrary constants, and $^{c}D_{0^{+}}^{\gamma
_{n-k}}$ is the standard Caputo fractional derivative of order $\gamma
_{n-k}.$ Now, by the Sadik transform, we get%
\begin{equation*}
\left[ \sum_{k=0}^{n}r_{n-k}\text{ }v^{\left( \gamma _{n-k}\right) \alpha }%
\right] \phi (v,\alpha ,\beta )=F(v,\alpha ,\beta )
\end{equation*}

The transfer function of fractional differential equation Eq.(\ref{12}) is
given by%
\begin{equation*}
K_{n}(v,\alpha ,\beta )=\frac{F(v,\alpha ,\beta )}{\phi (v,\alpha ,\beta )}=%
\left[ \sum_{k=0}^{n}r_{n-k}\text{ }v^{\left( \gamma _{n-k}\right) \alpha }%
\right] ^{-1}.
\end{equation*}

The unit-impulse response $\varphi _{i}(t)$ of the system is defined by the
inverse Sadik transform of $K_{n}(v,\alpha ,\beta )$ so that%
\begin{equation*}
\varphi _{i}(t)=\mathcal{S}^{-1}\left[ K_{n}(v,\alpha ,\beta )\right]
=k_{n}(t),
\end{equation*}

and the unit-step response function is given by the integral of $k_{n}(t)$
so that%
\begin{equation*}
\varphi _{s}(t)=I_{0+}^{1}k_{n}(t).
\end{equation*}

We give a simple example to illustrate the above system

\begin{example} \label{ex1}
\begin{figure}[tbp]
\includegraphics[width=\linewidth]{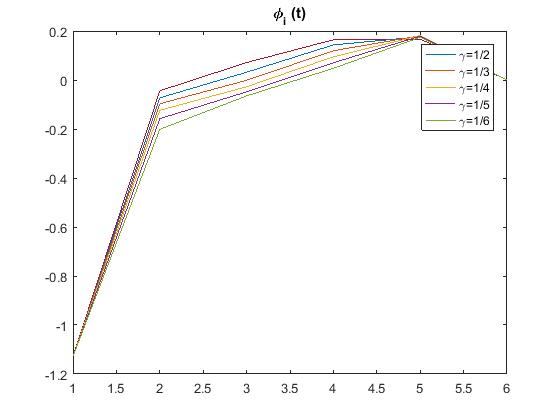}
\caption{Plot of unit-impulse response $\protect\varphi _{i}(t)$ in Example
\ref{ex1}}
\label{fig:boat2}
\end{figure}
We consider a simple fractional-order transfer function%
\begin{equation}
K_{2}(v,\alpha ,\beta )=(r\left( v^{\alpha }\right) ^{\gamma }+d)^{-1},\text{
\ }\gamma >0,  \label{eq12}
\end{equation}%
where $r$ and $d$ are arbitrary constants. The fractional differential
equation in the time domain identical to the transfer function (\ref{eq12})
is%
\begin{equation*}
r\text{ }^{c}D_{0^{+}}^{\gamma }\varphi (t)+d\varphi (t)=f(t),
\end{equation*}%
with the initial conditions%
\begin{equation*}
\varphi (0)=0.
\end{equation*}%
The unit-impulse response $\varphi _{i}(t)$ to the system is given by%
\begin{equation*}
\varphi _{i}(t)=\mathcal{S}^{-1}\left[ K_{2}(v,\alpha ,\beta )\right]
=k_{2}(t),
\end{equation*}%
where
\begin{equation*}
k_{2}(t)=\mathcal{S}^{-1}\left[ \frac{1}{r\left( v^{\alpha }\right) ^{\gamma
}+d}\right] =\frac{1}{r}\left[ t^{\gamma -1}E_{\gamma ,\gamma }(-\frac{d}{r}%
t^{\gamma })\right] ,
\end{equation*}

and the unit-step response to the system is%
\begin{eqnarray*}
\varphi _{s}(t) &=&\text{ }^{c}D_{0^{+}}^{-1}k_{2}(t)=I_{0+}^{1}\left( \frac{%
1}{r}\left[ t^{\gamma -1}E_{\gamma ,\gamma }(-\frac{d}{r}t^{\gamma })\right]
\right) \\
&=&\frac{1}{r}t^{\gamma }E_{\gamma ,\gamma +1}(-\frac{d}{r}t^{\gamma }).
\end{eqnarray*}

\begin{figure}[tbp]
\includegraphics[width=\linewidth]{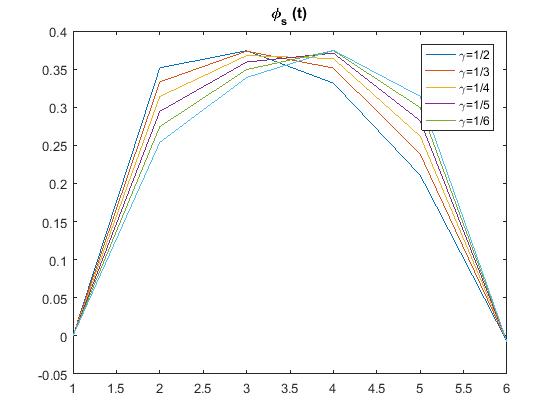} .
\caption{Plot of unit-step response $\protect\varphi _{s}(t)$ in Example \ref{ex1}}
\label{fig:boat3}
\end{figure}
\end{example}

\section*{conclusion}

There are a lot of the integral transforms of exponential type kernels, the
Sadik transform is new and very powerful among them and there are many
problems in engineering and applied sciences can be solved by Sadik
transform. So we have provided Sadik transform of Caputo fractional
derivative, and we also gave a sufficient condition to guarantee the
rationality of solving Caputo fractional differential equations by the Sadik
transform method. Moreover, The Sadik transform of time delay, initial value
theorem, and final value theorem are obtained. Some numerical examples to
justify our results are presented. An application of fractional-order
dynamical systems in control theory by Sadik transform is used. Finally, we
obtained some illustrative figures with the help of the Matlab software.

\end{document}